%% file: main.tex
\documentclass[12pt]{article}
\usepackage[sc]{mathpazo}
\usepackage[T1]{fontenc}
\usepackage[utf8]{inputenc}
\usepackage{microtype}
\usepackage{geometry}
\usepackage{titlesec}
\titleformat{\subsection}[runin]{\bf}{\thesubsection.}{3pt}{}
\usepackage{amsmath, amsfonts, amssymb, tikz-cd, mathtools}
\usepackage{graphicx}
\usepackage{caption, subcaption}
\usepackage{enumitem}
\DeclareMathOperator{\MCG}{\mathcal{MCG}}
\DeclareMathOperator{\PMCG}{\mathcal{PMCG}}
\DeclareMathOperator{\homeo}{Homeo}
\DeclareMathOperator{\Id}{Id}
\DeclareMathOperator{\Out}{Out}
\DeclareMathOperator{\cb}{CB}
\DeclareMathOperator{\rk}{rank}

\newcommand{\rkcb}{\rk_{\cb}}

\usepackage{amsthm}
\newtheorem{theorem}{Theorem}[section]
\newtheorem{lemma}[theorem]{Lemma}
\newtheorem{proposition}[theorem]{Proposition}
\newtheorem{corollary}[theorem]{Corollary}

\theoremstyle{definition}

\newtheorem{problem}[theorem]{Problem}

\theoremstyle{remark}
\newtheorem{remark}[theorem]{Remark}

\usepackage{bbm}
\usepackage{hyperref}
\definecolor{myPurp}{RGB}{105,73,148}
\definecolor{myRed}{RGB}{228,37,53}
\definecolor{myCyan}{RGB}{0,95,114}
\hypersetup{
     colorlinks=true,
     linkcolor = myRed,
     citecolor = myPurp,      
     urlcolor= myCyan,
     }
\newcommand{\qeda}{\hfill \ensuremath{\Box}}

\renewcommand\thesubsection{\thesection.\Alph{subsection}}

\usepackage{csquotes}

\numberwithin{equation}{section}

\newcommand\blfootnote[1]{%
  \begingroup
  \renewcommand\thefootnote{}\footnote{#1}%
  \addtocounter{footnote}{-1}%
  \endgroup
}

\title{{\bf Non-amenability of mapping class groups of infinite-type surfaces and graphs}}
\author{{\sc Yusen Long}}
\date{}

\begin{document}
\maketitle
\blfootnote{\emph{Date}: January 2026.}

\begin{abstract}
This paper completely determines the non-amenability of the mapping class groups of infinite-type surfaces, the mapping class groups of locally finite infinite graphs of higher ranks, gives an example of non-amenable stabiliser of a point at infinity of a coarsely bounded generated hyperbolic Polish group, and exhibits a class of mapping class groups of trees or rank-one graphs that are amenable.
\end{abstract}

\vspace{0.25cm}

\input{sections/section_1}
\input{sections/section_2}
\input{sections/section_3}
\input{sections/section_4}
\input{sections/section_5}

\bibliographystyle{alpha}
\bibliography{biblio}

\noindent{\sc Yusen Long}\\
\noindent{\sc Université Paris-Est Créteil, CNRS, LAMA UMR8050, F-94010 Créteil, France}

\noindent{\it Email address:} {\tt \href{mailto:yusen.long@u-pec.fr}{yusen.long@u-pec.fr}}

\end{document}

%% file: sections/section_1.tex
\section{Introduction}

In \cite{mccarthy1985tits} and \cite{ivanov1986algebraic}, McCarthy and Ivanov independently established a {\it Tits alternative} for subgroups in the mapping class groups of orientable finite-type surfaces, {\it viz.} any subgroup of the mapping class group \(\MCG(S)\) of a finite-type surface \(S\) either contains a solvable (actually, abelian) subgroup of finite index, or contains a non-abelian free group \(\mathbb{F}_n\) on \(n\geq 2\) generators. One of the central themes in geometric group theory is determining whether various groups satisfy the Tits alternative.

In the recent decades, mapping class groups of infinite-type surfaces (see \S \ref{s2b}), often called ``{\it big mapping class group}'', have received a lot of attention. A natural question that one can ask about big mapping class groups is whether they also have the Tits alternative.

The answer is negative: first, Funar and Kapoujian exhibited a big mapping class group containing Thompson's group \cite[Proposition 2.4]{funar2004universal} and thus not satisfying the Tits alternative. Other examples of big mapping class groups failing the Tits alternative were given in \cite{lanier2020centers} by Lanier and Loving as well as in \cite{aougab2021isometry} by Aougab, Patel and Vlamis, who showed that many big mapping class groups contain every countable group as discrete subgroup. We remark that, together non-amenability of some Burnside groups \cite{adian1979burnside,adyan1983random} and some finitely presented groups \cite{ol2003non}, this implies that big mapping class groups can admit non-amenable closed subgroup containing no non-abelian free group. More generally, Allock showed that most big mapping class groups fail to have the Tits alternative \cite{allcock2021most}.

Recall that a topological group is {\it amenable} if there exists an invariant probability measure on every compact space on which the group acts continuously. A classical result is that every virtually solvable group is an amenable discrete group, whereas any amenable discrete group cannot contain \(\mathbb{F}_2\) as a subgroup. Hence, the Tits alternative can also be thought of as a strengthening of the amenable/non-amenable dichotomy of the subgroups in the concerned group, equipped with discrete topology, {\it i.e.} a subgroup is non-amenable if and only if it contains non-abelian free groups.

Although the Tits alternative fails, one may still hope that a similar strengthening of amenable/non-amenable dichotomy in big mapping class group is true for Polish topology. In particular, we may still try to understand how amenable closed subgroups of big mapping class groups look like and to see if the big mapping class groups are themselves amenable or not.

It turns out that if a closed subgroup of a big mapping class group is amenable, then it must be nowhere dense. Equivalently, we have the following theorem:

\begin{theorem}\label{thm: big mcg}
Let \(S\) be an infinite-type surface. Then every open subgroup of \(\MCG(S)\) is not amenable. In particular, the group \(\MCG(S)\) itself is not amenable.
\end{theorem}

\begin{remark}
Although Theorem~\ref{thm: big mcg} is stated and proved for orientable surfaces, the proof goes verbatim for non-orientable surfaces.
\end{remark}

This is a stronger result than \cite{long2025big} which states that the big mapping class groups are not extremely amenable.

We remark that, as mentioned in \cite{long2025big}, there is an alternative geometric proof of this result when \(S\) is an infinite-type surface with a non-displaceable subsurface of finite type. In this case, one can construct a continuous action of \(\MCG(S)\) on a separable geodesic \(\delta\)-hyperbolic space \(X\) by isometries such that the action is of general type \cite{domat2022big, horbez2022big}. Since a continuous isometric action of an amenable group on such a space cannot be of general type \cite[Theorem 9.1.1]{long2024diverse}, this implies that \(\MCG(S)\) cannot be amenable.

Although geometric group theory is most commonly applied to finitely
generated groups, Rosendal \cite{rosendal2021coarse} observed
that the finiteness or compactness condition can be relaxed by using a notion of boundedness for general topological groups, and it has been seen that this allows us to extend the tools of geometric group
theory to a larger class of groups, namely \emph{coarsely bounded generated}, or, for abbreviation, \emph{CB} generated Polish group. For example, Mann and Rafi \cite{mann2023large} characterise mapping class groups that are CB generated, and hence have a canonical (up to
quasi-isometry) word metric. Moreover, there are CB generated big mapping class groups that is Gromov-hyperbolic \cite{schaffer2024graphs}. In view of these, in a recent note \cite{chandran2021infinite}, Chandran, Patel and Vlamis proposed the following problem:
\begin{problem}
Develop the theory of hyperbolic Polish groups.
\end{problem}

This problem is also related to the Tits alternative: it is a classical result from geometric group theory that finitely generated hyperbolic groups satisfies the Tits alternative. In particular, the stabiliser of a point at infinity on the Gromov boundary of a finitely generated hyperbolic group is amenable, see for example \cite[Chapitre 8]{ghys2013groupes}. However, this is not true in general for CB generated hyperbolic Polish group:

\begin{theorem}\label{thm: CB hyp}
There exists a CB generated hyperbolic Polish group such that there is a point on its Gromov boundary with a non-amenable stabiliser in it.
\end{theorem}
The example is given by a big mapping class group and Theorem~\ref{thm: CB hyp} is a consequence of Theorem~\ref{thm: big mcg}.

Outside of the surface world, a graph analogue of big mapping class groups is the mapping class group of a locally finite infinite graph. It can also be considered as the generalisation of the mapping class group of a finite graph, which is a discrete group isomorphic to \(\Out(\mathbb{F}_n)\), if the fundamental group of the finite graph is \(\mathbb{F}_n\,\). We note that similarly to mapping class groups of finite-type surfaces, \(\Out(\mathbb{F}_n)\) also satisfies the Tits alternative \cite{bestvina2000tits,bestvina2005tits}. Recently, Domat, Hoganson and Kwak have shown that the pure mapping class groups of finite type graphs also enjoy the Tits alternative, but not those of other graphs \cite{domat2023generating}. As a consequence, the Tits alternative also fails in general for the mapping class group of a locally finite infinite graph. So here, we ask about the amenability of these groups and of their closed subgroups.

The situation is much different from the case of big mapping class groups: here there are both amenable and non-amenable situations. We first prove:

\begin{theorem}\label{thm: big out fn}
Let \(\Gamma\) be a locally finite infinite graph of genus at least \(2\). Then \(\MCG(\Gamma)\) is not amenable. Moreover, if \(\Gamma\) is of infinite type, then every open subgroup of \(\MCG(\Gamma)\) is not amenable.
\end{theorem}

The proofs of Theorem~\ref{thm: big mcg} and Theorem~\ref{thm: big out fn} have a similar flavour. The idea is to find an open subgroup that admits continuous epimorphism onto a non-amenable discrete group, respectively the mapping class group of a finite type surface and \(\Out(\mathbb{F}_n)\,\).

When the underlying graph is a tree, the mapping class group is isomorphic as topological groups to the homeomorphism group of its ends \cite[Corollary 4.12]{bestvina2025groups}.

Recall that the {\it setwise stabiliser of \(K\)} is a subgroup given by 
\[\MCG(T)_K\coloneqq\left\{g\in\MCG(T)\colon g(K)=K\right\},\]
for any subset \(K\subset E\,\). Using Cantor-Bendixson analysis, we can show the following result about the amenability of the mapping class group of locally finite infinite trees:

\begin{theorem}\label{thm: homeo end}
Let \(T\) be a locally finite infinite tree. If \(\partial T\) is countable, then \(\MCG(T)\) is amenable. If there is a clopen subset \(K\subset \partial T\) of which the setwise stabiliser \(\MCG(T)_K\) admits no invariant probability measure on \(K\), then \(\MCG(T)\) is not amenable.
\end{theorem}

Together with Proposition~\ref{prop : 1 loop}, we can conclude from Theorem~\ref{thm: homeo end} the following result:
\begin{corollary}
Let \(\Gamma\) be a locally finite infinite graph with \(\rk(\Gamma)\leq 1\). If in addition \(\partial \Gamma\) is countable, then \(\MCG(\Gamma)\) is amenable.
\end{corollary}

We note that the latter case in Theorem~\ref{thm: homeo end} is abundant among the situations where \(E\) is uncountable, see Proposition~\ref{prop : many non-amen}. But this does not give a characterisation of the amenable mapping class groups of trees, nor of the rank-one graphs.

Since the end space of a tree is homeomorphic to a closed subset of the Cantor set, an equivalent way to state Theorem~\ref{thm: homeo end} without using the language of graph is the following:

\begin{theorem}
Let \(E\) be a closed subset of the Cantor set. Then \(\homeo(E)\) is amenable if \(E\) is countable. If there exists a clopen subset \(K\subset E\) of which the setwise stabiliser \(\homeo(E)_K\) admits no invariant probability measure on \(K\), then \(\homeo(E)\) is not amenable.\qeda
\end{theorem}

Finally, let us remark that both mapping class groups of infinite-type surfaces and mapping class groups of locally finite infinite graphs are non-Archemidean Polish groups (see \cite[Proposition2.1]{long2025big} and \cite[Proposition 4.7]{bestvina2025groups}). A complete characterisation of the amenability of these groups, as well as their closed subgroups, in terms of the Ramsey property is given in \cite{moore2013amenability}.

\subsection*{Acknowledgement.} The author is grateful to Federica Fanoni for encouraging in writing this paper. The author wants to thank Bruno Duchesne, Federica Fanoni, Matteo Tarocchi and Brian Udall for inspiring discussion and helpful comments. The author also thanks other feedback received on the Blooming Beasts conference and from the anonymous referees. The author is sponsored by the ANR grant Tremplin ERC-Starting Grant MAGIC (ANR-TERC-StG-0007).

%% file: sections/section_2.tex
\section{Preliminaries}\label{s2}

In this section, we will give the definition and several basic properties of amenable groups, as well as surfaces, graphs, and their mapping class groups.

\subsection{Amenability.}\label{s2a} Let \(G\) be a topological group. We say that \(G\) is {\it amenable} if every continuous action of \(G\) on a compact set admits a \(G\)-invariant probability measure.

This notion should be considered as an analytic or ergodic property of a topological group as it depends on the underlying group topology. For example, it is well known that being discrete, the non-abelian free group \(\mathbb{F}_2\) on \(2\) generators is not amenable, but it becomes amenable if it is equipped with the induced topology as a subgroup of the compact Lie group \(\mathrm{SO}(3)\,\). On the other hand, every group is amenable, if endowed with the trivial topology. Moreover, we remark that if a topological group is amenable, then every coarser topology will also render this group amenable.

However, in the rest of this paper, we will simply say that a group is amenable without explicitly mentioning its group topology if this topology is understood.

The class of amenable groups enjoys the following hereditary properties.
\begin{proposition}[Proposition 4.1, \cite{grigorchuk2017amenability}]
Let \(G\) be a topological group.
\begin{enumerate}[label=(A\arabic*), topsep=0pt, itemsep=-1ex, partopsep=1ex, parsep=1ex ]
    \item If \(G\) is amenable, then very open subgroup \(H<G\) is amenable. \label{HA1}
    \item If \(H\to G\) is a continuous homomorphism with dense image and if \(H\) is amenable, then \(G\) is amenable.\label{HA2}
    \item If \(G\) admits a dense subgroup that is a directed union of amenable subgroups, then \(G\) is amenable.\label{HA3}
    \item If \(G\) has an amenable closed normal subgroup \(N\) such that the quotient \(G/N\) is amenable, then \(G\) is amenable.\label{HA4}
    \item If \(G\) is amenable if and only if any dense subgroup \(N\) is amenable with respect to the induced topology.\label{HA5}
\end{enumerate}
\end{proposition}

Here we say that a topological group \(G\) is the {\it directed union} of a family \((H_\alpha)_{\alpha\in A}\) of subgroups in \(G\) if \(G=\bigcup_{\alpha\in A}H_\alpha\) is equipped with the inductive limit topology (with respect to inclusion) and if for any \(\alpha,\beta\in A\), there exists \(\lambda\in A\) such that \(H_\lambda\supset H_\alpha\cup H_\beta\,\).

We should remind the reader that, unlike the discrete cases, containing a free group \(\mathbb{F}_2\) does not imply that the concerned group is not amenable, even if it is a discrete subgroup. More generally, the openness condition in \ref{HA1} cannot be removed. An example is the unitary group \(\mathrm{U}(\mathcal{H}_\mathbb{C}^\infty)\) of an infinite dimensional separable complex Hilbert space \(\mathcal{H}_\mathbb{C}^\infty\,\). Equipped with the strong operator topology, the group \(\mathrm{U}(\mathcal{H}_\mathbb{C}^\infty)\) is amenable \cite[Proposition 5.2]{grigorchuk2017amenability} but contains every countable group as a discrete subgroup \cite[Lemma 5.1]{grigorchuk2017amenability}.

\subsection{Mapping class group of surfaces.}\label{s2b} Throughout this paper, we assume that the surfaces are orientable, connected, and without boundary components. Let \(S\) be such a surface. We say that \(S\) is of {\it finite type}, if the fundamental group \(\pi_1(S)\) is finitely generated, otherwise \(S\) will be of {\it infinite type}.

Another equivalent perspective to understand the difference between finite-type and infinite-type surfaces is to consider the space of {\it ends} and the cardinal of the genus.

The notion of end space was introduced by \cite{freudenthal1931enden} under the general setting of connected locally compact spaces \(X\,\). It is the space given by
\[\varprojlim_{K\subset X} \pi_0\big(X\setminus K\big)\,,\]
carrying the inverse limit topology, where the limit runs over all (connected) compact subsets \(K\subset X\,\). Thus, a surface is of infinite-type if and only if it has infinite genus or infinitely many ends.

The mapping class group of a surface \(S\), denoted by \(\MCG(S)\), is the quotient group of orientation-preserving homeomorphisms on \(S\) by homotopy, namely
\[\MCG(S)\coloneqq \homeo^+(S)/\homeo_0(S)\,,\]
where \(\homeo_0(S)\) is the connected component of \(\Id \in \homeo^+(S)\,\). If the surface \(\Sigma\) has a boundary component, then we will write \(\MCG(\Sigma,\partial \Sigma)\) for the subgroup of the mapping classes of orientation-preserving homeomorphisms fixing \(\partial \Sigma\) pointwise.

The topological group \(\homeo^+(S)\) has the compact-open topology, inducing the quotient topology on \(\MCG(S)\,\). One can easily see that the mapping class group \(\MCG(S)\) of a finite-type surface \(S\) is discrete. But when the surface \(S\) becomes infinite-type, the topology changes dramatically: it is not locally compact nor compactly generated \cite[Theorem 4.2]{aramayona2020big}. 

To be precise, the mapping class group \(\MCG(S)\) of an infinite-type surface \(S\) is actually isomorphic as topological group to a closed subgroup of the symmetry group \(\mathfrak{S}_\infty\) of countably infinite set, equipped with the pointwise convergence topology \cite[Proposition 2.1]{long2025big}. In particular, \cite{hernandez2018isomorphisms,bavard2020isomorphisms} implies that the topology on \( \MCG(S) \) is the permutation topology induced by its action on the set of isotopy classes of curves on \( S \), {\it i.e.} the topology is generated by the pointwise stabilisers of finitely many curves. This makes the topological group \(\MCG(S)\) a separable and completely metrisable space, {\it i.e.} a Polish group.

\subsection{Mapping class group of graphs.}\label{s2c} Let \(\Gamma\) be a locally finite, infinite, connected graph. The mapping class group of \(\Gamma\) can be thought of as a graph analogue to the mapping class group of a surface, whereas the definition is slightly different.

Recall that \(f\colon X \to Y\) is a {\it proper homotopy equivalence} if it is proper and there is another proper map \(g\colon Y \to X\) such that both \(f\circ g\) and \(g\circ f\) are properly homotopic to the identity. The mapping class group of \(\Gamma\), denoted by \(\MCG(\Gamma)\), is the group of proper homotopy equivalences of \(\Gamma\) up to proper homotopy. We remark that this group should not be confused with the group of proper homotopy equivalences of \(\Gamma\), which is a proper extension of the mapping class group, nor should it be confused with the group of \emph{homotopy equivalences} which does not require the properness condition, see \cite{bestvina2025groups} for further details.

Similarly to the mapping class groups of infinite-type surfaces, the group \(\MCG(\Gamma)\) is also equipped with a Polish group topology. This group topology is obtained by the quotient of compact-open topology of inverse pairs of proper homotopy equivalences by proper homotopy. See \cite[\S 4]{bestvina2025groups} for more details.

Since \(\Gamma\) is connected and locally compact, we can also define the end space of \(\Gamma\), denoted by \(\partial \Gamma\,\). We say that \(\Gamma\) is of {\it infinite type} if \(\pi_1(\Gamma)\) is not finitely generated and the end space is not finite, otherwise it is of {\it finite type}.

The mapping class group of the graph \(\Gamma\) also acts continuously on the end space \(\partial(\Gamma)\) and thus gives rise to a homomorphism \(\MCG(\Gamma)\to\homeo(\partial(\Gamma))\) and we define the {\it pure mapping class group} \(\PMCG(\Gamma)\) to be the kernel of this homomorphism. For more details, see \cite{bestvina2025groups,domat2023coarse}.

If the graph is finite, then \(\MCG(\Gamma)\) is isomorphic to the discrete group \(\Out(\mathbb{F}_n)\), where \(\mathbb{F}_n\) is the free group on \(n\) generators and \(n\geq 0\) is the {\it rank} of the graph, {\it i.e.} an integer such that \(\pi_1(\Gamma)\simeq \mathbb{F}_n\,\). 

If \(\Gamma\) is a locally finite infinite graph with trivial fundamental group, then it becomes a locally finite tree. In this case, the mapping class group \(\MCG(\Gamma)\) will be isomorphic as topological groups to the homeomorphism group \(\homeo(\partial \Gamma)\) of the end space \(\partial \Gamma\), equipped with compact-open topology \cite[Corollary 4.12]{bestvina2025groups}.

It is a classical result that the end space \(\partial(\Gamma)\) is homeomorphic to a closed subset of the Cantor set, so the mapping class group of a tree is isomorphic to the homeomorphism group of a closed subset of the Cantor set.

%% file: sections/section_3.tex
\section{Mapping class groups of infinite-type surfaces}\label{s3}

Let \( S \) be an infinite-type surface, and let \( \MCG(S) \) denote its mapping class group. Throughout this text, we will not distinguish between a proper embedding of the unit circle \( S^1 \) into \( S \) and its isotopy class and both will be referred to as a {\it curve} on \( S \). For a collection \(F\) of curves on \(S\), we will write
\[\MCG(S)_{(F)}\coloneqq \{g\in \MCG(S)\colon g(\alpha)=\alpha \text{ for all }\alpha\in F\}\]
the {\it pointwise stabiliser} of \(F\,\).

We begin with the following foundational fact.

\begin{lemma}\label{lem_fin_col}
Let \( S \) be an infinite-type surface. Then there exists a finite collection of pairwise disjoint simple closed curves \( F = \{c_1, \ldots, c_n\} \) and a connected, non-sporadic, finite-type subsurface \( \Sigma \subset S \) such that \( \partial \Sigma = c_1 \cup \dots \cup c_n \), and for any \( g \in \MCG(S)_{(F)} \) in the pointwise stabilizer of \( F \), we have \( g(\Sigma) = \Sigma \,\).
\end{lemma}

\begin{proof}
It suffices to take a connected, non-sporadic, finite-type essential subsurface \( \Sigma \subset S\) such that no complementary component of \(\Sigma\) is homeomorphic to itself and take \(F=\partial \Sigma\,\). This is always possible by choosing a subsurface without finite-type complementary component. Since any \( g \in \MCG(S)_{(F)} \) in the pointwise stabilizer of \( F \) permutes the complementary components \(S\setminus F\) and \(\Sigma\) is the only one of finite type, we can conclude \(g(\Sigma)=\Sigma\,\).
\end{proof}

Now we are able to show Theorem~\ref{thm: big mcg}. The idea is to use Lemma~\ref{lem_fin_col} to construct a non-amenable open subgroup in each open subgroup of \(\MCG(S)\,\).
\begin{proof}[Proof of Theorem~\ref{thm: big mcg}]
Let \( G < \MCG(S) \) be any open subgroup. By the definition of the permutation topology, there exists a finite collection \(F\) of curves on \(S\) such that \(\MCG(S)_{(F)}<G\,\). As the union of curves in \(F\) is compact, their complementary component contains an infinite-type subsurface \(S'\subset S\,\). By Lemma~\ref{lem_fin_col}, we can find a connected, non-sporadic, finite-type essential subsurface \(\Sigma\subset S'\) and another finite collection \(F'=\partial \Sigma\) of curves on \(S'\) such that \(g(\Sigma)=\Sigma\) for any \(g\in \MCG(S)_{(F\cup F')}<\MCG(S)_{(F')}\,\). This allows us to build a group homomorphism
\[\phi\colon \MCG(S)_{(F\cup F')}\to\MCG(\Sigma,\partial\Sigma)\]
by restricting a mapping class to \(\Sigma\,\). Using the definition of the compact-open topology and its quotient topology, it is not difficult to see that \(\phi\) is a continuous group homomorphism. Moreover, the group \(\MCG(\Sigma,\partial\Sigma)\) injects in \(\MCG(S)_{(F\cup F')}\) since \(F\cap \Sigma=\emptyset\), which allows us to deduce that \(\phi\) is an epimorphism. But \(\MCG(\Sigma,\partial \Sigma)\) is not amenable, in particular it is a discrete group containing a non-abelian free group of rank \(2\) \cite[Theorem 3.14]{farb2011primer}. By \ref{HA2}, the open subgroup \(\MCG(S)_{(F\cup F')}\) is not amenable. Now \ref{HA1} will yield that \(G\) itself is not amenable.
\end{proof}

As we have already seen, Theorem~\ref{thm: big mcg} can also be stated in the following way: all amenable subgroups in the big mapping class group have to be nowhere dense. However, this does not imply that such a group is discrete. For example, we can take the subgroup \(G<\MCG(S)\) generated by the Dehn twists along infinitely many pairwise disjoint curves. It is easy to see that \(G\) is isomorphic to \(\mathbb{Z}^\infty\) and is an abelian group, so it is amenable, but has an accumulation point at \(\Id\in\MCG(S)\), so it is not discrete.

\begin{proof}[Proof of Theorem~\ref{thm: CB hyp}]
Let \(S\coloneqq \mathbf{S}^2\setminus (\mathcal{C}\cup\{p\})\), where \(\mathbf{S}^2\) is the \(2\)-dimensional sphere, \(\mathcal{C}\) is a Cantor set and \(p\in \mathbf{S}^2\setminus \mathcal{C}\) is an isolated puncture. It has been already seen in \cite{schaffer2024graphs} that \(\MCG(S)\) is CB generated hyperbolic Polish group with infinite diameter and is quasi-isometric to the loop graph \(L(S;p)\). Consider a non-compact finite type subsurface \(V\subset S\) with \(p\) being an end of \(V\). Let \(\varphi\in \MCG(S)\) be such that \(\varphi|_V\) is pseudo-Anosov and \(\varphi|_{S\setminus V}\) is identity. Since \(\varphi\) is loxodromic on  \(L(S;p)\) by \cite[Theorem~1.1]{rasmussen2021wwpd}, it is also loxodromic on \(\MCG(S)\). Let \(\xi\) be a point on the Gromov boundary of \(\MCG(S)\) fixed by \(\varphi\). By Lemma~\ref{lem_fin_col}, we can find an open subgroup \(H<\MCG(S)\) fixing \(V\) pointwise. We can check that \(H\) commutes with \(\varphi\), so it is contained in the stabiliser of \(\xi\), \emph{i.e.} the stabiliser of \(\xi\) is open and thus not amenable by Theorem~\ref{thm: big mcg}.
\end{proof}

%% file: sections/section_4.tex
\section{Mapping class groups of infinite-type graphs}\label{s4}

Let \(\Gamma\) be a locally finite infinite-type graph, and let \(\MCG(\Gamma)\) be its mapping class group. Suppose now that \(\Gamma\) is not a tree. Recall that a homotopy \(H\colon C(X,X)\times [0,1]\to C(X,X)\) between two continuous maps \(u=H(u,0)\colon X\to X\) and \(v=H(u,1)\colon X\to X\) is {\it stationary on \(K\subset X\)} if the restriction \(H(u,t)|_K=\Id_K\) is the identity map on \(K\), for all \(t\in[0,1]\,\). Let \(K\subset \Gamma\) be a finite subgraph and define \(U_K\) to be the set of properly homotopy equivalence classes \([f]\in\MCG(\Gamma)\) with a representative such that
\begin{enumerate}[label=(\arabic*), topsep=0pt, itemsep=-1ex, partopsep=1ex, parsep=1ex ]
\item \(f=\Id\) on \(K\),\label{cond.1}
\item \(f\) preserves each complementary component of \(K\),\label{cond.2}
\item there is a representative \(g\) of \([f]^{-1}\in\MCG(\Gamma)\) verifying \ref{cond.1} and \ref{cond.2},
\item there are proper homotopies \(gf\simeq \Id\) and \(fg\simeq \Id\) that are stationary on \(K\) and preserve complementary components of \(K\,\).
\end{enumerate}
In view of \cite[Proposition 4.7]{bestvina2025groups}, the subgroups \(U_K\) are open in \(\MCG(\Gamma)\) and form a neighbourhood basis of \(\Id\in\MCG(\Gamma)\,\).

Similarly to the proof of Theorem~\ref{thm: big mcg}, we can construct some \(U_K\) that is non-amenable.
\begin{lemma}\label{lem: build subgraph}
Let \(\Gamma\) be a locally finite infinite-type graph with rank at least \(2\). Suppose that \(K\subset \Gamma\) is a finite subgraph so that only one complementary component \(K'\) is a finite subgraph of \(\Gamma\) but not a tree. Then \(U_K\) is not amenable.
\end{lemma}
\begin{proof}
By the definition of \(U_K\), the restriction map to \(K'\) yields a group homomorphism
\[\psi\colon U_K\to \MCG(K')\]
and this homomorphism is continuous since \(\MCG(K')\) is discrete by definition and \(\ker(\psi)\supset U_{K\cup K'}\) is open. We note that the (properly) homotopy classes of homeomorphisms \(K'\to K'\) are also contained in \(U_K\), which gives us inclusion \(\MCG(K')\hookrightarrow U_{K}\,\). This implies that \(\psi\) is an epimorphism. However, there exists a natural homomorphism \(\MCG(K')\to \Out(\pi_1(K'))\simeq \Out(\pi_1(\mathbb{F}_n))\), where \(n\geq 2\) is the rank of \(K'\). But \(\Out(\mathbb{F}_n)\) contains non-abelian free subgroups (see, for example, \cite[\S 3.4]{bestvina2000tits}) as a discrete group, so it is not amenable. By \ref{HA2}, we can conclude that the topological group \(U_K\) is not amenable as it continuously maps onto \(\Out(\mathbb{F}_n)\) via \(\psi\,\).
\end{proof}

To show that \(\MCG(\Gamma)\) or its open subgroup \(G<\MCG(\Gamma)\) is not amenable, it suffices to find a finite subgraph \(K\subset\Gamma\) verifying the assumptions in Lemma~\ref{lem: build subgraph}. 

\begin{proof}[Proof of Theorem~\ref{thm: big out fn}]
Indeed, start with a finite subgraph \(K'\subset \Gamma\) of rank at least \(2\). Let \(V\) be a compact neighbourhood of \(K'\) such that no complementary component of \(V\) is compact. This \(V\) always exists since \(\Gamma\) is locally finite, the total number of complementary components of \(V\) is finite, and we can take a new \(V\) to be the union of \(V\) and all its compact complementary components, which will remain a compact neighbourhood of \(K'\,\). Let \(K\) be the finite graph \(V\setminus K'\,\). The \(K\) will be a finite subgraph in \(\Gamma\) satisfying the assumptions in Lemma~\ref{lem: build subgraph}. Thus, \(U_K\) is not amenable, forcing \(\MCG(\Gamma)\) to be non-amenable in view of \ref{HA1}.

Suppose now \(\Gamma\) is of infinite type. Let \(G\) be an arbitrary open subgroup of \(\MCG(\Gamma)\,\). By \cite[Proposition 4.7]{bestvina2025groups}, there exists a finite subgraph \(K_0\subset \Gamma\) such that \(U_{K_0}<G\,\). Now take \(K'\subset \Gamma\setminus K_0\) to be finite non-tree subgraph, which is always feasible since \(\Gamma\) is infinite-type. Similarly to the case above, we can construct a compact neighbourhood \(V\) of \(K'\cup K_0\) such that \(K\coloneqq V\setminus K'\) satisfies the assumptions of Lemma~\ref{lem: build subgraph}. Note that \(U_{K}<U_{K_0}<G\) as \(K_0\subset K\,\). Lemma~\ref{lem: build subgraph} implies that \(U_{K}\) is a non-amenable open subgroup in \(G\). Hence, we can conclude the non-amenability of \(G\) by appealing to \ref{HA1}.
\end{proof}

When the graph \(\Gamma\) is not infinite-type, the mapping class group \(\MCG(\Gamma)\) might contain amenable open subgroups. Recall that the {\it core} \(\Gamma_g\) of \(\Gamma\) is the smallest subgraph that has positive rank. Observe that if \(\Gamma\) is neither a tree nor of infinite type, then \(\Gamma_g\) is a finite subgraph of \(\Gamma\). In this case, \(U_{\Gamma_g}\) is an open subgroup of \(\MCG(\Gamma)\) and can be identified with the mapping class group of the tree \(\overline{\Gamma\setminus\Gamma_g}\subset\Gamma\,\). However, we will see in \S~\ref{s5} that the mapping class group of some trees is amenable.

Finally, we conclude that if the graph \(\Gamma\) has only one immersed loop, then the amenability of \(\MCG(\Gamma)\) is completely determined by that of the mapping class group of this tree, {\it i.e.} \(\MCG(\Gamma)\) is amenable if and only if \(\MCG(T)\) is.
\begin{proposition}\label{prop : 1 loop}
Let \(\Gamma\) be a locally finite infinite graph of rank \(1\). Let \(T\coloneqq \overline{\Gamma\setminus\Gamma_g}\subset \Gamma\,\). Then \(\MCG(\Gamma)\) is amenable if and only if \(\MCG(T)\) is.
\end{proposition}
\begin{proof}
Note that \(\MCG(T)\simeq \homeo(\partial T)\) and \(\homeo(\partial T)=\MCG(\Gamma)/\PMCG(\Gamma)\). To conclude the desired result, by \ref{HA4} and \ref{HA5}, it suffices to show that the normal closed subgroup \(\PMCG(\Gamma)\) of \(\MCG(\Gamma)\) is amenable. By \cite[Corollary 3.9]{bestvina2025groups}, we have
\[\PMCG(\Gamma)\simeq \mathcal{R}\rtimes \PMCG(\Gamma_g^\ast),\]
where \(\mathcal{R}\) is a group consisting of certain functions \(\partial \Gamma\to\mathbb{Z}\) and \(\Gamma_g^\ast\) is a rank \(1\) graph with one end. Hence, the group \(\mathcal{R}\) is abelian and \(\PMCG(\Gamma_g^\ast)\simeq \mathbb{Z}_2\). These imply that \(\PMCG(\Gamma)\) is in fact discretely amenable and thus amenable for any topology.
\end{proof}

%% file: sections/section_5.tex
\section{Mapping class groups of trees}\label{s5}

Let \(T\) be a locally finite infinite tree. Recall that the end space \(E\coloneqq\partial T\) is homeomorphic to a closed subset of the Cantor set. Moreover, if we equip \(\homeo(E)\) with compact-open topology, then \cite[Corollary 4.12]{bestvina2025groups} states that the topological groups \(\MCG(T)\simeq \homeo(E)\) are isomorphic.

Hence, studying the amenability of the mapping class group of a tree \(T\) is equivalent to understanding the same problem for the group of homeomorphisms of a closed subset \(E\) of the Cantor set. In view of this, it is convenient to appeal to Cantor-Bendixson analysis, for which more details can be found in \cite[\S I.6]{kechris2012classical}.

By the Cantor-Bendixson Theorem, we have the following decomposition
\[E=C\cup D\, ,\]
where \(C\) is the {\it perfect kernel} of \(E\) and is homeomorphic to the Cantor set if non-void, and where \(D\) is an at most countable subset. We remark that the above statements hold for any locally finite infinite graph (not just for trees).

We recall that the {\it derived set} \(X'\) of a topological space \(X\) is the set of the points in \(X\) that are a limit point of \(X\). More generally, we define for each ordinal
\begin{align*}
    X^0 &= X,\\
    X^{(\alpha+1)} &=\big(X^{(\alpha)}\big)',\\
    X^{(\alpha)} &= \bigcap_{\beta<\alpha} X^{(\beta)},\text{ if }\alpha\text{ is limit.}
\end{align*}
Thus \((X^{(\alpha)})_{\alpha\geq 0}\) is a decreasing transfinite sequence of closed subsets of \(X\).

Moreover, we can define the {\it Cantor-Bendixson rank} of a point \(x\in X\), denoted by \(\rkcb (x,X)\), to be the largest ordinal \(\alpha\) such that \(x\in X^{(\alpha)}\). With this notion, we can further decompose
\[D=\bigcup_{\alpha\geq 0}D_\alpha\]
such that each \(D_\alpha\) is the collection of the points in \(E\) of Cantor-Bendixson rank \(\alpha\). Since the Cantor-Bendixson rank is invariant under homeomorphisms, it follows that \(f(F_\alpha)=D_\alpha\) for every \(f\in \homeo(E)\) and any countable ordinal \(\alpha\geq 0\).

Moreover, there exists some countable \(\alpha_0\geq 0\) such that \(D_{\alpha}=\emptyset\) for any \(\alpha>\alpha_0\,\).This countable ordinal \(\alpha_0\) is the unique ordinal such that \(D_{\alpha_0}\) is finite. This \(\alpha_0\) is the {\it Cantor-Bendixson rank} of \(D\) and is denoted by \(\rkcb (D)\). Note that \(\rkcb (D)=\sup_{x\in D}\rkcb (x,D)\,\). For every \(\beta>\alpha_0=\rkcb(D)\), we have
\[E^{(\beta)}=C\, .\]

For each countable ordinal \(\alpha\geq 0\), we define \(G_\alpha<\homeo(E)\simeq \MCG(T)\) as the pointwise stabiliser of \(E^{(\alpha)}\). In particular, we note that if \(\alpha<\beta\), then \(G_\alpha\subset G_\beta\). Equivalently, \(G_\alpha\) consists of homeomorphisms of \(E\) supported on \(\bigcup_{\lambda<\alpha}D_\lambda\).

Recall that the compact-open topology on \(\homeo(E)\) is such that a basis of neighbourhoods of the identity is the collection \(\big(V_\mathcal{P}\big)_{\mathcal{P}}\), where \(\mathcal{P}\) runs over the finite partition \(\mathcal{P}\) of \(E\) into clopen subsets and \(V_\mathcal{P}\) os the clopen subgroup of elements leaving every \(P_i\in\mathcal{P}\) invariant. See for example \cite[\S 4.1]{bestvina2025groups}.

Now we are able to treat the amenability of the subgroups \(G_\alpha<\homeo(E)\simeq\MCG(T)\) by transfinite induction.

To start with, let us consider the following axillary lemma.

For any countable ordinal \(\alpha>0\), define a group homomorphism \(\rho_\alpha\colon \homeo(E)\to \homeo(E^{(\alpha)})\) as the restriction map to \(E^{(\alpha)}\). Set \(G^\sharp_\alpha\coloneqq\rho_\alpha(G_{\alpha+1})<\homeo(E^{(\alpha)})\).

\begin{lemma}\label{lemma : quotient}
Let \(\alpha>0\) be a countable ordinal. Then group \(G_{\alpha+1}/G_\alpha\) is isomorphic to \(G^\sharp_\alpha\) as topological groups.
\end{lemma}
\begin{proof}
First, we need to note that the closed subgroup \(G_\alpha\) is a normal subgroup of \(G_{\alpha+1}\) so that \(G_{\alpha+1}/G_\alpha\) is a topological group. Indeed, as a pointwise stabiliser of a closed subset, \(G_\alpha\) is defined by a closed condition and is a closed subgroup. Moreover, given any \(g\in G_\alpha\), any \(f\in G_{\alpha+1}\), and any \(x\in E^{(\alpha)}\), since \(f^{-1}(x)\in E^{(\alpha)}\) by the definition of \(G_{\alpha+1}\), we can conclude that \(fgf^{-1}(x)=f\circ g\big(f^{-1}(x)\big)=f\circ f^{-1}(x)=x\), {\it i.e.} we have \(fgf^{-1}\in G_\alpha\).

By definition, the restriction \(\rho_\alpha\colon G_{\alpha+1}\to G^\sharp_\alpha\) is surjective. Note that \(\ker(\rho_\alpha)=G_\alpha\) by definition. Hence \(\rho_\alpha\) yields a group isomorphism \(\phi_\alpha\colon G_{\alpha+1}/G_\alpha \to G^\sharp_\alpha\,\).

It is not difficult to see that the restriction map \(\rho_\alpha\) is continuous, which further implies that the quotient map \(\phi_\alpha\) is also continuous. To show that \(\phi_\alpha\) is a homeomorphism, it suffices to show that \(\phi_\alpha\) is also open. Take a finite partition \(\mathcal{P}\coloneqq(P_i)\) of \(E\) into clopen subsets. Then \(\mathcal{P}_\alpha\coloneqq(P_i\cap E^{(\alpha)})_i\) will become a finite partition of \(E^{(\alpha)}\) into clopen subsets. By the definition of quotient topology, a basis of neighbourhoods in \(\homeo(E)/G_\alpha\) of the identity is the subgroup of classes \([\widehat{f}]\) such that \(\widehat{f}\in V_\mathcal{P}\). Let \(f\coloneqq \phi_\alpha([\widehat{f}])\in\homeo(E^{(\alpha)})\). Then \(f=\rho_\alpha(\widehat{f})\), which implies that
\[f(P_i\cap E^{(\alpha)})=\widehat{f}(P_i\cap E^{(\alpha)})=\widehat{f}(P_i)\cap E^{(\alpha)}=P_i\cap E^{(\alpha)},\]
for all \(P_i\in\mathcal{P}\), {\it i.e.} \(f\) leaves each element in \(\mathcal{P}_{\alpha}\) invariant. Equivalently, with the definition of compact-open topology on \(\homeo(E^{(\alpha)})\), we have shown that the homomorphism \(\phi_\alpha\) is also open.
\end{proof}

Let us treat the case where \(\alpha=1\).

\begin{lemma}\label{lemma : g1}
Let \(E\) and \(G_1\) be as above. Then \(G_1\) is amenable.
\end{lemma}
\begin{proof}
We remark that \(G_1\) is the subgroup of homeomorphisms of \(E\) that are only supported on the discrete subspace of isolated points \(D_0\). If \(|D_0|<\infty\), then \(G_1\) is finite and thus is amenable. Suppose now that \(|D_0|=\infty\). Let \(F\subset D_0\) be an arbitrary finite subset and let \(\mathfrak{S}_F\) be the symmetry group of \(F\). We note that \(\mathfrak{S}_F<G_1\) by definition and that \(\mathfrak{S}_F\) is amenable as a finite subgroup. Now consider the subgroup
\[H\coloneqq \bigcup_{\substack{F\subset D_0\\ |F|<\infty}}\mathfrak{S}_F<G_1\]
of finitely supported homeomorphisms. It is not difficult to verify that \(H\) is a directed union of amenable groups, and by \ref{HA3} it is itself (discretely) amenable. We claim that \(H\) is a dense subgroup of \(G_1\) and, by \ref{HA5}, it follows that \(G_1\) is also amenable. Indeed, given any finite partition \(\mathcal{P}=(P_i)_i\) of \(E\) into clopen subsets, it yields a finite partition \(\widetilde{\mathcal{P}}=(\widetilde{P}_i)_i\) of \(D_0\), with \(\widetilde{P}_i\coloneqq P_i\cap D_0\), and the subgroup
\[V_\mathcal{P}\cap G_1=\{g\in G_1\colon g P_i=P_i\,,\ \forall P_i\in \mathcal{P}\}=\{g\in G_1\colon g \widetilde{P}_i=\widetilde{P}_i\,,\ \forall \widetilde{P}_i\in \widetilde{\mathcal{P}}\}\eqqcolon W_{\widetilde{\mathcal{P}}}\]
forms a basis of neighbourhoods of the identity in \(G_1\). By the pigeonhole principle, there exists \(\widetilde{P}_i\in\widetilde{\mathcal{P}}\) such that \(|\widetilde{P}_{i_0}|>1\). Hence, we can take \(F\subset \widetilde{P}_{i_0}\) with \(|F|>1\) so that there exists a \(\sigma\in\mathfrak{S}_F\subset H\) different from the identity. Since \(\sigma\) is only supported on \(F\subset \widetilde{P}_{i_0}\), it also leaves each \(\widetilde{P}_i\) invariant. This implies that \(\Id\neq\sigma\in W_{\widetilde{\mathcal{P}}}<G_1\) and that \(H\) is dense in \(G_1\) by consequence.
\end{proof}

The following axillary lemma will allow us to prove the induction steps:
\begin{lemma}\label{lem : G sharp}
For countable ordinal \(\alpha\geq0\), the topological group \(G^\sharp_\alpha\) is amenable.
\end{lemma}
\begin{proof}
We note that the subspace topology on \(D_\alpha\) is discrete, otherwise there exists a point in \(D_\alpha\) that is an accumulation point of \(D_\alpha\) itself, which contradicts the definition of Cantor-Bendixson rank. So for each \(p\neq q\in D_\alpha\), one can find a clopen neighbourhood \(U\ni p\) and \(V\ni q\) in \(E\) so that \(U\cap V=\emptyset\) and \(U\simeq V\simeq \omega^\alpha+1\). In turn, we can construct a homeomorphism \(f\colon E\to E\) such that \(f(V)=U\) with \(f(p)=f(q)\) and \(f(x)=x\) for every \(x\notin U\cup V\,\). Hence, we can conclude that the support
\[\mathrm{supp}(G^\sharp_\alpha)\coloneqq \{x\in E^{(\alpha)}:\exists g\in G^\sharp_\alpha\text{ such that }gx\neq x\}\]
is the full set of isolated points in \(E^{(\alpha)}\), {\it i.e.} the full set of \(D_\alpha\,\). Then applying the same arguments in Lemma~\ref{lemma : g1} to \(E^{(\alpha)}\) and \(G^\sharp_\alpha\) will give the amenability of \(G^\sharp_\alpha\,\).
\end{proof}

For the inductive steps, we first treat the case of a limit ordinal.
\begin{lemma}\label{lemma : limit ordinal}
Let \(E\), \(G_\alpha\) be given as above. Suppose that \(\lambda\geq 0\) is a countable limit ordinal and that \(G_\alpha\) is amenable for all \(\alpha<\lambda\,\). Then \(G_\lambda\) is also amenable.
\end{lemma}
\begin{proof}
It suffices to note that \(G_\lambda=\bigcup_{\alpha<\lambda}G_\alpha\) is the directed union of amenable subgroups \(G_\alpha\) and apply \ref{HA3}.
\end{proof}

For the successor ordinals, we shall take the quotient to deduce the amenability of \(G_{\alpha+1}\):
\begin{lemma}\label{lemma : successor ordinal}
Let \(E\), \(G_\alpha\) be as above. Suppose that \(G_{\alpha}\) is amenable. Then so is \(G_{\alpha+1}\).
\end{lemma}
\begin{proof}
In view of \ref{HA4}, we only need to show that \(G_{\alpha+1}/G_\alpha\) is amenable. By Lemma~\ref{lemma : quotient}, the group \(G_{\alpha+1}/G_\alpha\) is isomorphic to the toplogical group \(G^\sharp_\alpha\), which is amenable following Lemma~\ref{lem : G sharp}.
\end{proof}

With the results obtained above, transfinite induction immediately yields the following:
\begin{proposition}\label{prop : g_a}
Let \(E\) be a closed subset of the Cantor set and let \(G_\alpha<\homeo(E)\) be the pointwise stabiliser of \(E^{(\alpha)}\) for countable ordinal \(\alpha\geq0\). Then \(G_\alpha\) is amenable. \qeda
\end{proposition}

Using the amenability of \(G_\alpha\) and the Cantor-Bendixson decomposition of \(E\) when \(E\) is countable, we can deduce the amenability of the mapping class group of trees with countably many ends. For trees with uncountable endspace, we can also show that many of them have a non-amenable mapping class group.

\begin{proof}[Proof of Theorem~\ref{thm: homeo end}]
Suppose that \(E\) is countable. This is equivalent to \(C=\emptyset\), or \(E=D\,\). Then the Cantor-Bendixson decomposition yields \(D=\bigcup_{0\leq \lambda\leq\alpha_0} D_\lambda\) for some countable ordinal \(\alpha_0>0\). By definition, it immediately follows that \(\homeo(E)=G_\alpha\) for any \(\alpha>\alpha_0\,\). Then Proposition~\ref{prop : g_a} implies that \(\homeo(E)\) is amenable, or equivalent \(\MCG(T)\) of a locally finite infinite tree with countable end space \(\partial T\) is amenable.

Suppose now that there exists a clopen subset \(K\subset E\) such that the setwise stabiliser of \(K\)
\[\homeo(E)_K\coloneqq\left\{f\in\homeo(E)\colon f(K)=K\right\}\]
does not have an invariant probability measure on the compact set \(K\). This implies that \(\homeo(E)_K\) is not amenable. However, \(\homeo(E)_K<\homeo(E)\) is an open subgroup for the compact-open topology: \(\homeo(E)_K=V_\mathcal{P}\) with a finite partition of \(E\) into clopen subsets \(\{K,E\setminus K\}\,\). Then by \ref{HA1}, \(\homeo(E)\) is not amenable.
\end{proof}

\begin{remark}
We note that if there exists a clopen subset \(K\) of which the setwise stabiliser \(\homeo(E)_K\) admits no invariant probability measure on \(K\), then necessarily \(E\) is uncountable and has a non-empty perfect kernel homeomorphic to the Cantor set.
\end{remark}

An example of an uncountable closed subset \(E\) of the Cantor set such that \(\homeo(E)\) is not amenable is the case when \(E\) is the full Cantor set. For every clopen subset \(A\subset E\), we can find homeomorphisms \(f,g\in\homeo(E)\) such that \(f(A),g(A)\subset A\) but \(f(A)\cap g(A)=\emptyset\,\). Hence, if \(\mu\) is a \(\homeo(E)\)-invariant measure on \(E\), then necessarily
\[\mu(A)\geq \mu\big(f(A)\big)+\mu\big(g(A)\big)=2\mu(A),\]
forcing \(\mu(A)=0\), which implies that \(\mu=0\) and that it is not a probability measure, {\it i.e.} the group \(\homeo(E)\) does not admit an invariant probability measure on \(E\) and thus is not amenable.

With this in mind, we can construct many examples of uncountable closed subset \(E\) of the Cantor set such that \(\homeo(E)\) is not amenable:
\begin{proposition}\label{prop : many non-amen}
Let \(E\) be an uncountable closed subset of the Cantor set with the Cantor-Bendixson decomposition \(E=C\cup D\,\). Suppose that \(\overline{D}\cap C\neq C\,\). Then \(\homeo(E)\) is not amenable.
\end{proposition}
\begin{proof}
By our assumption, the subset \(A\coloneqq C\setminus\overline{D}\subset C\) is a non-empty open subset of the Cantor set \(C\). By a classical result \cite{schoenfeld1975alternate}, this subset \(A\) is either homeomorphic to \(C\) itself or homeomorphic to \(C\setminus\{x\}\) for any \(x\in C\,\). In either case, there exists a clopen subset \(K\subset A\) that is homeomorphic to the Cantor set. Moreover, by extending the homeomorphisms on \(K\) by identity on \(E\setminus K\,\), we can show that the action of \(\homeo(E)\) on \(K\) maps surjectively via restriction to the actions of the open subgroup \(\homeo(E)_K\) on \(K\). Since \(K\) is homeomorphic to the Cantor set, it follows from our discussion above that \(\homeo(K)\,\), and thus \(\homeo(E)_K\) does not admit an invariant probability measure on the compact set \(K\). This implies that \(\homeo(E)_K\) is not amenable, forcing \(\homeo(E)\) to be non-amenable by \ref{HA1}.
\end{proof}

It is worth remarking that Proposition~\ref{prop : many non-amen} is not a characterisation of non-amenable cases, since there exists an uncountable closed subset \(E\) of the Cantor set with Cantor-Bendixson decomposition \(E=C\cup D\) such that \(\overline{D}\cap C=C\), yet \(\homeo(E)\) is not amenable. It suffices to consider the compact space \(E\) of the union of the Cantor set \(C\) and a countable set of isolated points \(D\) such that each point of the Cantor set is an accumulation point of the isolated points. Then \(\overline{D}\cap C=C\) but \(\homeo(E)\) is not amenable as it does not admit an invariant probability measure on \(C\).

Using the partial order on the end space \(E\) from \cite{mann2023large}, for a maximal end \(x\in E\), we write \(E(x)\) for the ends that are equivalent to \(x\). If there exists a maximal end \(x\in E\) such that \(x\) is stable and \(E(x)\) is homeomorphic to the Cantor set, then \(\homeo(E)\) acts on the clopen subset \(E(x)\subset E\) transitively (see for example \cite{mann2024two}), and thus does not have an invariant probability measure on \(E(x)\), which implies the non-amenability of \(\homeo(E)\) by Theorem~\ref{thm: homeo end}. Hence, by the dichotomy of maximal ends established in \cite[Proposition 4.7]{mann2023large}, we can further conclude:
\begin{proposition}\label{prop: finite maximal ends}
Let \(E\) be as above. If \(\homeo(E)\) is amenable, then for every maximal stable end \(x\in E\), the number of ends of the same maximal type \(E(x)\) is finite. \qed
\end{proposition}
We end the discussion by remarking that if \(E\) is countable, then there are only finitely many equivalent maximal stable ends, which also falls into the category in Proposition~\ref{prop: finite maximal ends}.